\documentclass[12pt]{article}
\usepackage{latexsym,amssymb,amsmath,amsfonts,pictex,enumerate,amsthm,dsfont}
\usepackage{tikz}
\usetikzlibrary{decorations.markings}
\usepackage{scalerel,stackengine}
\usepackage[active]{srcltx}
\usepackage{ gensymb }

\theoremstyle{definition}
\newtheorem{theorem}{Theorem}

\newtheorem{lemma}[theorem]{Lemma}

\stackMath
\newcommand\reallywidehat[1]{%
\savestack{\tmpbox}{\stretchto{%
  \scaleto{%
    \scalerel*[\widthof{\ensuremath{#1}}]{\kern-.6pt\bigwedge\kern-.6pt}%
    {\rule[-\textheight/2]{1ex}{\textheight}}
  }{\textheight}%
}{0.5ex}}%
\stackon[1pt]{#1}{\tmpbox}%
}
\begin{document}

 \title{Bounded point derivations and functions of bounded mean oscillation}


\author{Stephen Deterding \thanks{email: Stephen.Deterding@westliberty.edu} \\
West Liberty University,\\
West Liberty, WV, USA}




\date{}

\maketitle

\begin{abstract}
Let $X$ be a subset of the complex plane and let $A_0(X)$ denote the space of VMO functions that are analytic on $X$. $A_0(X)$ is said to admit a bounded point derivation of order $t$ at a point $x_0 \in \partial X$ if there exists a constant $C$ such that $|f^{(t)}(x_0)|\leq C ||f||_{BMO}$ for all functions in $VMO(X)$ that are analytic on $X \cup \{x_0\}$. In this paper, we give necessary and sufficient conditions in terms of lower $1$-dimensional Hausdorff content for $A_0(X)$ to admit a bounded point derivation at $x_0$. These conditions are similar to conditions for the existence of bounded point derivations on other functions spaces.
\end{abstract}

\section{Introduction}

Let $X$ be a subset of the complex plane and suppose that $f$ is analytic on a neighborhood of $X$. If $x_0$ is an interior point of $X$, then it follows from the Cauchy estimates that $f'(x_0)$ cannot be too large relative to the size of $f$. To be precise, in this situation there exists a positive number $C$ such that for every function $f$ analytic on $X$, $|f'(x_0)| \leq C ||f||_{\infty}$, where $||f||_{\infty}$ is the supremum norm of $f$ on $X$. However, this is not true if $x_0$ is a boundary point. For example, if $X$ is the unit disk then the sequence $f_n = z^n$ has supremum norm 1 on $X$ for all $n$ but $f_n'(1) \to \infty$ as $n\to \infty$.

\bigskip

Runge's theorem states that every analytic function on a neighborhood of $X$ can be approximated uniformly by rational functions with poles off $X$, so we can suppose that the function $f$ is a rational function with poles off $X$. Let $R(X)$ denote the closure of the rational functions with poles off $X$ in the uniform norm. For non-negative integer values of $t$ we say that $R(X)$ admits a bounded point derivation of order $t$ at $x_0$ if there exists a constant $C$ such that for all rational functions $f$ with poles off $X$, $|f^{(t)}(x_0)| \leq C ||f||_{\infty}$. From the above discussion, we see that $R(X)$ admits a bounded point derivation at all interior points of $X$, while at a boundary point $x_0$ the derivatives of analytic functions are bounded in the uniform norm if and only if $R(X)$ admits a bounded point derivation at $x_0$. 

\bigskip

One can define bounded point derivations for the closure of rational functions in other Banach spaces as well. If $B$ is a Banach space with norm $||\cdot||_B$ and $R_B(X)$ denotes the closure of rational functions in $B$ then we say that $R_B(X)$ admits a bounded point derivation of order $t$ at $x_0$ if there exists a constant $C$ such that for all rational functions $f$ with poles off $X$, $|f^{(t)}(x_0)| \leq C ||f||_{B}$. 

\bigskip

Necessary and sufficient conditions for the existence of bounded point derivations have been determined for the closure of rational functions in the following Banach spaces $B$: $C(X)$, $L^P(X)$, Lip$\alpha(X)$, and in this paper we will determine these conditions for $VMO(X)$, the space of functions of vanishing mean oscillation. Our main theorem is the following.

\bigskip

\begin{theorem}
\label{BMO}

Let $X$ be a subset of the complex plane and let $A_0(X)$ denote the space of VMO functions that are analytic on $X$. Choose $x_0 \in \partial X$ and let $A_n = \{2^{-(n+1)} \leq |z-x_0| \leq 2^{-n}\}$. Then there exists a bounded point derivation of order $t$ on $A_0(X)$ at $x_0$ if and only if 

\begin{equation*}
    \sum_{n=1}^{\infty} 2^{(t+1)n} M^1_*(A_n \setminus X) < \infty,
\end{equation*}

\bigskip

\noindent where $M^1_*$ denotes lower 1-dimensional Hausdorff content.

\end{theorem}

\bigskip

In the next section, we review some of the results of bounded point derivations on other function spaces to show how the conditions for bounded point derivations on $A_0(X)$ compare with those for the other Banach spaces. In section 3 we review the concepts of BMO spaces and Hausdorff contents, and in section 4 we prove Theorem \ref{BMO}. In the last section, Theorem \ref{BMO} is used to provide examples of sets with and without bounded point derivations on $A_0(X)$.

\section{Bounded Point Derivations on Other Function Spaces}

Let $U$ be an open subset of the complex plane and let $0 < \alpha < 1$. A function $f : U \to \mathbb{C}$ satisfies a Lipschitz condition with exponent $\alpha$ on $U$ if there exists $k>0$ such that for all $z,w \in U$
\begin{equation}
\label{lip condition}
|f(z) - f(w)| \leq k |z-w|^{\alpha}.
\end{equation}

\bigskip

Let Lip$\alpha(U)$ denote the space of functions that satisfy a Lipschitz condition with exponent $\alpha$ on $U$. Lip$\alpha(U)$ is a Banach space with norm given by $\displaystyle ||f||_{Lip\alpha(U)} = \sup_{U} |f| + k(f)$, where $k(f)$ is the smallest constant that satisfies \eqref{lip condition}. 
\bigskip

An important subspace of Lip$\alpha(U)$ is the little Lipschitz class, lip$\alpha(U)$, which consists of those functions in Lip$\alpha(U)$ such that 

\begin{equation*}
    \lim_{\delta \to 0^+} \sup_{0<|z-w|<\delta} \dfrac{|f(z)-f(w)|}{|z-w|^{\alpha}} = 0
\end{equation*}

\bigskip

\noindent

Let $A_{\alpha}(U)$ denote the space of lip$\alpha$ functions that are analytic on $U$. $A_{\alpha}(U)$ is said to admit a bounded point derivation of order $t$ at $x_0 \in \partial X$ if there exists a constant $C$ such that

\begin{equation*}
    |f^{(t)}(x_0)| \leq C ||f||_{Lip\alpha}
\end{equation*}

\bigskip

\noindent whenever $f \in$ lip$\alpha(\mathbb{C})$ is analytic in a neighborhood of $U \cup \{x_0\}$. In \cite{Lord} Lord and O'Farrell determined necessary and sufficient conditions for $A_{\alpha}(U)$ to admit a bounded point derivation at $x_0$ in terms of lower dimensional Hausdorff content, which is defined in the next section. In that paper, what we call $A_{\alpha}(U)$ is denoted by $a_{\alpha}(U)$ since $A_{\alpha}(U)$ is used to denote the space of Lip$\alpha$ functions that are analytic on $U$. Their result is the following \cite[Theorem 1.2]{Lord}.

\begin{theorem}
\label{Lord}
There exists a bounded point derivation of order $t$ on $A_{\alpha}(U)$ at $x_0$ if and only if 

\begin{equation*}
    \sum_{n=1}^{\infty} 2^{(t+1)n} M_*^{1+\alpha}(A_n \setminus U) < \infty,
\end{equation*}

\bigskip

\noindent where $M_*^{1+\alpha}$ denotes lower $(1+\alpha)$-dimensional Hausdorff content.

\end{theorem}

Theorem \ref{BMO} and Theorem \ref{Lord} both involve series of Hausdorff contents of annuli. These are similar to existence theorems for bounded point derivations on other function spaces such as Hallstrom's theorem for $R(X)$ \cite[Theorem 1, 1$'$]{Hallstrom}, Hedberg's theorem for $R^p(X)$, $p>2$ \cite[Theorem 2]{Hedberg}, and the theorem of Fernstrom and Polking for $R^2(X)$\cite[Theorem 6]{Fernstrom}. In addition, O'Farrell has recently proven a similar theorem to Theorem \ref{Lord} \cite[Theorem 3.7]{O'Farrell2019} for functions belonging to negative Lipschitz classes in which $-1<\alpha<0$. Notably, this excludes the case of $\alpha =0$. As we will see the case of $\alpha =0$ can be identified with the space of analytic VMO functions.

\section{Bounded Mean Oscillation and Hausdorff Content}

Let $f \in L^1_{loc}(\mathbb{C})$ and let $Q$ be a cube in the complex plane with area $|Q|$. The mean value of $f$ on $Q$, denoted by $f_Q$, is 

\begin{equation*}
    f_Q = \frac{1}{|Q|} \int_Q f(z) dA
\end{equation*}
\bigskip

\noindent and the mean oscillation of $f$ on a cube $Q$, denoted by $\Omega(f,Q)$, is 

\begin{equation*}
    \Omega(f,Q) = \frac{1}{|Q|} \int_Q |f(z)-f_Q| dA.
\end{equation*}

\bigskip

\noindent The function $f$ is said to be of bounded mean oscillation if 

\begin{equation*}
  ||f||_* =  \sup_Q \Omega(f,Q) < \infty,
\end{equation*}

\bigskip

\noindent where the supremum is taken over all cubes $Q$ in $\mathbb{C}$. Let $BMO(\mathbb{C})$ denote the set of functions of bounded mean oscillation and let $BMO(X) = \{f|_X : f \in BMO(\mathbb{C})\}$. Let $||f||_X = \displaystyle \inf ||f||_*$, where the infimum is taken over all functions $F$ such that $F =f$ on $X$. $||f||_X$ is a seminorm on $BMO(X)$, which vanishes only at the constant functions. If we let $\displaystyle ||f||_{BMO(X)} = ||f||_X + \left|\int_X f(z) dA\right|$, then $||f||_{BMO(X)}$ defines a norm on $BMO(X)$.

\bigskip
 An important subspace of $BMO(X)$ is $VMO(X)$, the space of functions of vanishing mean oscillation. For $f$ in $BMO(X)$ and $\delta >0$ let 

\begin{equation*}
    \Omega_f(\delta) = \sup_Q \{\Omega(f,Q): \text{ radius Q } \leq \delta\}.
\end{equation*}

\bigskip

\noindent  $VMO(X)$ consists of those functions in $BMO(X)$ which satisfy $\Omega_f(\delta) \to 0$ as $\delta \to 0^+$ for all cubes $Q \in \mathbb{C}$. Let $A_0(X)$ denote the space of restrictions to $X$ of $VMO$ functions which are analytic on a neighborhood of $X$. Alternately, $f$ is a function of bounded mean oscillation on a set $X$ if and only if there exists a constant $K$ such that for every cube $Q \subseteq X$ and some constant $c_Q$, the inequality

\begin{equation*}
 \int_{Q}|f(z) -c_Q| dA \leq K r^2  
\end{equation*}

\bigskip

\noindent holds where $r$ is the length of the edge of $Q$ \cite{John}. This is similar to an alternate characterization of Lip$\alpha$ functions. $f \in $ Lip$\alpha(X)$ if and only if there exists a constant $K$ such that for every cube $Q \subseteq X$ and some constant $c_Q$, the inequality

\begin{equation*}
 \int_{Q}|f(z) -c_Q| dA \leq K r^{2+\alpha}  
\end{equation*}

\bigskip

\noindent
holds where $r$ is the length of the edge of $Q$ \cite{Meyers}. In this way $A_0(X)$ can be seen as the limit as $\alpha \to 0$ of $A_{\alpha}(X)$.  Thus we say that $A_0(X)$ admits a bounded point derivation of order $t$ at $x_0 \in \partial X$ if there exists a constant $C$ such that

\begin{equation*}
    |f^{(k)}(x_0)| \leq C ||f||_{BMO(X)}
\end{equation*}

\bigskip

\noindent whenever $f \in VMO(\mathbb{C})$ is analytic in a neighborhood of $X \cup \{x_0\}$. 

\bigskip

We saw in the last section that bounded point derivations on $A_{\alpha}(U)$ are characterized by lower $(1+\alpha)$ dimensional Hausdorff content, which suggests that bounded point derivations on $A_0(X)$ are characterized by lower $1$-dimensional Hausdorff content. Some other examples of the connection between $VMO$ and lower $1$-dimensional Hausdorff content in the context of rational approximation can be found in the paper of Verdera \cite{Verdera} (Here what we call $A_0(X)$ is denoted by $VMO_a(X)$), the paper of Boivin and Verdera \cite{Boivin}, and the paper of Bonilla and Fari\~{n}a \cite{Bonilla}.  

\bigskip

We now define the lower $1$-dimensional Hausdorff content. A measure function is an increasing function $h(t)$, $t\geq 0$, such that $h(t) \to 0$ as $t \to 0$. If $h$ is a measure function then define

\begin{equation*}
    M^h(E) = \inf \sum_j h(r_j),
\end{equation*}

\bigskip

\noindent where the infimum is taken over all countable coverings of $E$ by squares with sides of length $r_j$. The lower $1$-dimensional Hausdorff content of $E$, denoted $M_*^1(E)$ is defined by 

\begin{equation*}
  M_*^1(E)= \sup M^h(E) 
\end{equation*}

\bigskip

\noindent where the supremum is taken over all measure functions $h$ such that $h(t) \leq t$ and $h(t)t^{-1} \to 0$ as $t\to 0^+$. Furthermore, the infimum can be taken over countable coverings of $E$ by dyadic squares \cite[pg. 61, Lemma 1.4]{Garnett}. It also follows from the definition that lower 1-dimensional Hausdorff content is subadditive; that is, $U \subseteq V$ implies $M_*^1(U) \leq M_*^1(V)$.

\section{The main result}

For the proof of Theorem \ref{BMO}, we will need the following lemma.

\begin{lemma}
\label{lem1}

Let $E_j$, $j= 1, \ldots, n$ be sets and let $E = \displaystyle \bigcup_{j=1}^{n} E_j$. Let $\displaystyle \psi = \sum_{j=1}^n \psi_j$, where each $\psi_j$ has support on $E_j$. Then

\begin{equation*}
    \int_E \psi dA = \sum_{j=1}^n \int_{E_j} \psi_j dA.
\end{equation*}

\end{lemma}

\begin{proof}

We first prove the case of $n=2$.

\begin{align*}
    \int_E \psi dA &= \int_{E_1 \setminus E_2} \psi dA + \int_{E_2 \setminus E_1} \psi dA + \int_{E_1 \cap E_2} \psi dA\\
    &= \int_{E_1 \setminus E_2} \psi_1 dA + \int_{E_2 \setminus E_1} \psi_2 dA + \int_{E_1 \cap E_2} \psi_1 dA + \int_{E_1 \cap E_2} \psi_2 dA\\
    &= \int_{E_1} \psi_1 dA + \int_{E_2} \psi_2 dA.
\end{align*}

\bigskip

\noindent For the general case, let $F = \displaystyle \bigcup_{j=1}^{n-1}E_j$ and let $\psi_F = \displaystyle \sum_{j=1}^{n-1} \psi_j$. Then 

\begin{equation*}
    \int_E \psi dA = \int_F \psi_F dA + \int_{E_n} \psi_n dA, \end{equation*}
    
    \bigskip
    
    \noindent and the lemma follows by induction.

\end{proof}

We now prove Theorem \ref{BMO}.

\begin{proof}

Choose $f \in A_0(X)$ such that $f$ is analytic on $X$ and suppose that $||f||_* \leq 1$. For each $n$ let $K_n$ be a compact subset of $A_n \setminus X$ such that $f$ is analytic on $A_n \setminus K_n$. Since $f$ has a finite number of poles, we only need a finite number of $K_n$. Fix $n$ and let $\{Q_j\}$ be a covering of $K_n$ by dyadic squares so that no squares overlap except at their boundaries. Let $r_j$ denote the side length of $Q_j$. Let $Q_j^* = \frac{3}{2}Q_j$, the square with side length $\frac{3}{2}r_j$ and the same center as $Q_j$, and let $D_n = \bigcup Q_j^*$. Then by the Cauchy integral formula

\begin{align*}
    f^{(t)}(x_0) = \frac{t!}{2\pi i} \sum_n \int_{\partial D_n} \dfrac{f(z)}{(z-x_0)^{t+1}}dz.\\
\end{align*}

\bigskip

\noindent For each individual square $Q_j$, we can construct a smooth function $\phi_{j}$ such that $ \phi_{j}$ has support on $Q_j^*$, $||\nabla \phi_{j}||_{\infty} \leq C r_j^{-1}$, and $\sum_j \phi_j = 1$ on a neighborhood of $\bigcup \phi_j$. Such a construction can be found in \cite[Lemma 3.1]{Harvey}. Let $\displaystyle \phi = 1- \sum_j \phi_j$. It then follows from Green's theorem and Lemma \ref{lem1} that

\begin{align*}
    \left|\frac{t!}{2\pi i}  \int_{\partial D_n} \dfrac{f(z)}{(z-x_0)^{t+1}} dz\right| &= \left|\frac{t!}{2\pi i} \int_{\partial D_n} \dfrac{f(z) \phi(z)}{(z-x_0)^{t+1}}dz\right|\\
    &= \left|\frac{t!}{\pi} \int_{ D_n} \dfrac{f(z) }{(z-x_0)^{t+1}} \dfrac{\partial \phi}{\partial \overline{z}}dA\right|\\
    &\leq \frac{t!}{\pi} \sum_j \left|\int_{Q_j^*} \dfrac{f(z) }{(z-x_0)^{t+1}} \dfrac{\partial \phi_j}{\partial \overline{z}}dA\right|.\\
\end{align*}

\bigskip

\noindent Since $\int_{Q_j^*}(z-x_0)^{-(t+1)}\dfrac{\partial \phi_j}{\partial \overline{z}} dA= \int_{\partial Q_j^*}(z-x_0)^{-(t+1)} \phi_j(z) dz = 0 $, it follows that

\begin{align*}
      \frac{t!}{\pi} \sum_j \left|\int_{Q_j^*} \dfrac{f(z) }{(z-x_0)^{t+1}} \dfrac{\partial \phi_j}{\partial \overline{z}}dA\right|
    & \leq C \sum_j \int_{Q_j^*} |f(z) - f_{Q_j^*}| \cdot |z-x_0|^{-(t+1)}\left|\dfrac{\partial \phi}{\partial \overline{z}}\right| dA  \\
    &\leq C 2^{(t+1)n} \sum_j  m(Q_j^*) \Omega(f, Q_j^*) r_j^{-1} \\
    &\leq C 2^{(t+1)n} \sum_j r_j \Omega_f \left(\frac{3}{2}r_j\right).   \\
\end{align*}

\bigskip

\noindent Since, $f \in VMO(\mathbb{C})$, the measure function $h(t) = t\Omega_f \left(\frac{3}{2}t\right)$ satisfies the conditions in the definition of $M^1_*$. Hence by taking the infimum over all such covers $\{Q_j\}$, we have that

\begin{equation*}
    \left|\frac{t!}{2\pi i}  \int_{\partial D_n} \dfrac{f(z)}{(z-x_0)^{t+1}} dz\right| \leq  C 2^{(t+1)n} M_*^1(K_n) 
\end{equation*}

\bigskip

\noindent and since $M_*^1$ is subadditive, it follows that

\begin{align*}
    |f^{(t)}(x_0)| &\leq C  \sum_{n=1}^{\infty} 2^{(t+1)n} M_*^1( K_n)\\
    &\leq C  \sum_{n=1}^{\infty} 2^{(t+1)n} M_*^1(A_n \setminus X)\\
    &\leq C.\\
\end{align*}

\bigskip

\noindent If $g(z) \in VMO(\mathbb{C})$ is analytic on $X \cup \{x_0\}$, let $f(z) = \dfrac{g(z)}{||g||_*}$. Then $||f||_* \leq 1$ and hence $|f^{(t)}(x_0)| \leq C.$ Thus $|g^{(t)}(x_0)| \leq C||g||_*$ and $A_0(X)$ admits a bounded point derivation of order $t$ at $x_0$.

\bigskip

To prove the converse, we can assume that $x_0 = 0$ and $X$ is entirely contained in the unit disk, and we suppose that 

\begin{equation*}
    \sum_{n=1}^{\infty} 2^{(t+1)n}M_*^1(A_n \setminus X) = \infty,
\end{equation*}

\bigskip

\noindent and we choose a decreasing sequence $\epsilon_n$ such that 

\begin{equation*}
    \sum_{n=1}^{\infty} 2^{(t+1)n}\epsilon_n M_*^1(A_n \setminus X) = \infty
\end{equation*}

\bigskip

\noindent and $|2^{(t+1)n}\epsilon_n M_*^1(A_n \setminus X)| \leq 1$ for all $n$.

\bigskip

We now modify a construction used by Lord and O'Farrell for approximation in Lipschitz norms \cite[pg.12]{Lord}. It follows from Frostman's Lemma that for each $n$ there exists a positive measure $\nu_n$ with support on $A_n \setminus X$ such that

\begin{enumerate}
\item $\nu_n(B(z,r)) \leq \epsilon_n r$ for all balls $B$,

    \item $\int \nu_n = C \epsilon_n M_*^1(A_n \setminus X)$.
\end{enumerate}

\bigskip

Let

\begin{equation*}
    f_n(z) = \int \left(\dfrac{\zeta}{|\zeta|}\right)^{t+1} \dfrac{d\nu_n(\zeta)}{\zeta -z}
\end{equation*}

\bigskip

\noindent The same argument used in the proof of (b) of \cite{Kaufman} (See also \cite[pg.288]{Verdera}.) shows that $f_n$ is analytic off $A_n$, $||f_n||_* \leq C \epsilon_n$ and $f_n \in A_0(X)$. In addition,

\begin{equation*}
    f_n^{(k)}(0) = t! \int \dfrac{d\nu_n(\zeta)}{|\zeta|^{t+1}}
\end{equation*}

\bigskip

\noindent Hence $f_n^{(t)}(0) \geq C 2^{(t+1)n} \epsilon_n M^1_*(A_n \setminus X)$. For each $n$ choose $p>n$ such that 

\begin{equation*}
   1 \leq \sum_{m=n}^p 2^{(t+1)m} \epsilon_m M^1_*(A_m \setminus X) \leq 2 
\end{equation*}

\bigskip

\noindent and let 

\begin{equation*}
    g_n(z) = \sum_{m=n}^p f_m(z).
\end{equation*}

\bigskip
\noindent It follows that $g_n^{(t)}(0)$ is bounded below by a nonzero constant for all $n$. We wish to show that $||g_n||_{BMO(X)} \to 0$ as $n\to \infty$. 

\bigskip

Let $Q$ be a cube in the annulus $A_k = \{2^{-(k+1)} \leq |z| \leq 2^{-k}\}$ and choose $f_m$ with $n \leq m \leq p$. Let $f_{m,Q} = \frac{1}{|Q|} \int_Q f_m dA$. Then there are 3 cases.

\begin{enumerate}
    \item $k = m-1$, $m$, or $m+1$.
    \item $k \geq m+2$
    \item $k \leq m-2$.
\end{enumerate}

\bigskip

If $k = m-1$, $m$, or $m+1$ then 

\begin{equation*}
\Omega(f_m, Q) = \dfrac{1}{|Q|} \int_Q |f_m(z) - f_{m,Q}|dA_z \leq ||f_m||_* \leq C \epsilon_n.
\end{equation*}
\bigskip

If $k \geq m+2$ then 

\begin{align*}
\dfrac{1}{|Q|} \int_Q |f_m(z) - f_{m,Q}|dA_z &\leq 2||f_m||_{L^{\infty}(A_k)}\\
& \leq 2 \cdot 2^{m+2} \int \nu_m\\
& \leq C 2^m \epsilon_m M^1_*(A_m \setminus X).
\end{align*}

\bigskip

If $k \leq m-2$ then 

\begin{align*}
\dfrac{1}{|Q|} \int_Q |f_m(z) - f_{m,Q}|dA_z &\leq 2||f_m||_{L^{\infty}(A_k)}\\
& \leq 2 \cdot 2^{m-2} \int \nu_m\\
& \leq C 2^m \epsilon_m M^1_*(A_m \setminus X).
\end{align*}

\bigskip

It follows from the triangle inequality that 

\begin{align*}
    ||g_n||_* &\leq \sum_{m=n}^p ||f_m||_*\\
    &\leq 3C\epsilon_n + C \sum_{m=n}^p 2^m \epsilon_m M^1_*(A_m \setminus X)\\
    & \leq 3C\epsilon_n + C2^{-n} \sum_{m=n}^p 4^m \epsilon_m M^1_*(A_m \setminus X)\\
    &\leq C (3\epsilon_n + 2^{-n}).
\end{align*}

\bigskip

\noindent Moreover from Fubini's theorem,

\begin{align*}
    ||f_m||_{L^1(X)} &= \int_X \left| \int \left(\dfrac{\zeta}{|\zeta|}\right)^{t+1} \dfrac{d\nu_m(\zeta)}{\zeta -z}\right| dA(z)\\
    &\leq \int_X \int \dfrac{d\nu_m(\zeta)}{|\zeta -z|}dA\\
    &= \int \int_X \dfrac{1}{|\zeta -z|}dA(z) d\nu_m(\zeta)
\end{align*}

\bigskip

\noindent Fix $\zeta$ and let $B_r$ denote the ball with radius $r$ centered at $\zeta$, so that the area of $B_r$ is the same as the area of $X$. Then 

\begin{align*}
    \int \int_X \dfrac{1}{|\zeta -z|}dA(z) d\nu_m(\zeta) &\leq \int \int_{B_r} \dfrac{1}{|\zeta -z|}dA(z) d\nu_m(\zeta)\\
    &\leq 2 \pi r \int d\nu_m(\zeta)\\
    & \leq C \epsilon_m M^1_*(A_m \setminus X) 
\end{align*}

\bigskip

\noindent Thus, $\displaystyle ||g_n||_{L^1(X)} \leq C \sum_{m =n}^p \epsilon_m M^1_*(A_m \setminus X) \leq C$. Hence $||g_n||_{BMO(X)} \to 0$ as $n\to \infty$, but $g_n^{(t)}(0)$ is bounded away from $0$ for all $n$. Hence $A_0(X)$ does not admit a bounded point derivation at $0$.

\end{proof}

\section{Examples}

In this section, we will construct some examples for which Theorem \ref{BMO} applies. The first construction is that of a set $X$ such that $A_0(X)$ does not admit a bounded point derivation at $0$ but such that $A_{\alpha}(X)$ admits a bounded point derivation at $0$ for all $\alpha >0$.

\bigskip

Let $D$ be the open unit disk and let $A_n = \{2^{-(n+1)}\leq |z|\leq 2^{-n}\}$. For each $n$, let $B_n$ be a closed disk contained entirely in $A_n$ with radius $r_n = 4^{-n} n^{-1}$ and let $X = D \setminus \bigcup B_n$. This is known as a roadrunner set. See figure \ref{roadrunner}.

\begin{figure}[t]
\begin{center}
\begin{tikzpicture}[scale=1]

\tikzset{->-/.style={decoration={
  markings,
  mark=at position .5 with {\arrow[scale=2]{>}}},postaction={decorate}}}

     
     
   \draw[black] (0,0) circle (3cm);
   \draw [black](2,0) circle (.5cm);
   \draw [black](1,0) circle (.25cm);
   \draw [black](.5,0) circle (.125cm);
   \filldraw[fill=black, draw=black] (0,0) circle (.08 cm);
   \node [below] (x) at (0, 0){$0$};

\end{tikzpicture}
\caption{A roadrunner set}
\label{roadrunner}

\end{center}
\end{figure}
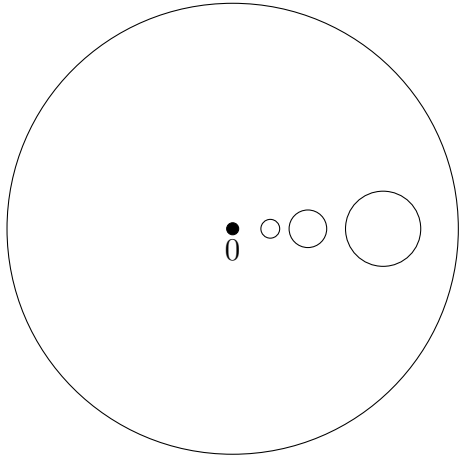

Since $M^1_*(A_n \setminus X) = M^1_*(B_n) = r_n$, it follows that

\begin{equation*}
    \sum_{n=1}^{\infty} 4^n M^1_*(A_n \setminus X) = \sum_{n=1}^{\infty} n^{-1} = \infty,
\end{equation*}

\bigskip

\noindent and thus $A_0(X)$ does not admit a bounded point derivation at $0$. However, 
$M^{1+\alpha}_*(A_n \setminus X) = M^{1+\alpha}_*(B_n) = r_n^{1+\alpha}$. Hence 

\begin{equation*}
    \sum_{n=1}^{\infty} 4^n M^{1+\alpha}_*(A_n \setminus X) = \sum_{n=1}^{\infty} 4^{-n\alpha} n^{-(1+\alpha)} < \infty,
\end{equation*}

\bigskip

\noindent and thus $A_{\alpha}(X)$ admits a bounded point derivation at $0$ for all $\alpha >0$.

\bigskip For an example of a set for which $A_0(X)$ admits a bounded point derivation at $0$, we can modify the previous construction so that the removed disks $B_n$ have radii $r_n = 4^{-n} n^{-2}$. Then

\begin{equation*}
    \sum_{n=1}^{\infty} 4^n M_*^1(A_n \setminus X) = \sum_{n=1}^{\infty} n^{-2} < \infty.
\end{equation*}

\bibliographystyle{amsplain}

\end{document}